\documentclass{amsart}
\usepackage{graphicx}
\usepackage{epsfig}
\usepackage[usenames,dvipsnames]{color}

\newtheorem{thm}{Theorem}[section]

\newtheorem{lem}[thm]{Lemma}

\theoremstyle{definition}
\newtheorem{defn}[thm]{Definition}
\theoremstyle{remark}
\newtheorem{rem}[thm]{Remark}
\numberwithin{equation}{section}
\newtheorem{example}[thm]{Example}
\newcommand{\set}[1]{\left\{#1\right\}}
\newcommand{\Real}{\mathbb R}
\newcommand{\Z}{\text{$\mathbb{Z}$}}

\newcommand{\N}{\text{$\mathbb N$}}

\newcommand{\To}{\longrightarrow}

\newcommand{\paren}[1]{\left(#1\right)}

\newcommand{\fmonio}{\tilde{f}_i^\ell}
\newcommand{\fil}{f_i^\ell}

\begin{document}

\title[Dual fusion frames]{Dual fusion frames}

\author[S. B. Heineken]{S. B. Heineken}
\address{%
Departamento de Matem\'atica, FCEyN, Universidad de Buenos Aires and
IMAS-CONICET,
\\ Pabell\'on I, Ciudad Universitaria,
C1428EGA C.A.B.A., Argentina.}
\email{sheinek@dm.uba.ar}%

\author[P. M. Morillas]{P. M. Morillas}
\address{%
Instituto de Matem\'{a}tica Aplicada San Luis, UNSL-CONICET and
Departamento de
Matem\'{a}tica, FCFMyN, UNSL,\\
         Ej\'{e}rcito de los Andes 950, 5700 San Luis, Argentina.}%
\email{morillas@unsl.edu.ar}%

\author[A. M. Benavente]{A. M. Benavente}
\address{%
Instituto de Matem\'{a}tica Aplicada San Luis, UNSL-CONICET and Departamento de Matem\'{a}tica, FCFMyN, UNSL,\\ Ej\'{e}rcito de los Andes 950, 5700 San Luis, Argentina.}%
\email{abenaven@unsl.edu.ar}%

\author[M. I. Zakowicz]{M. I. Zakowicz}
\address{%
Departamento de Matem\'{a}tica, FCFMyN, UNSL,\\ Ej\'{e}rcito de los Andes 950, 5700 San Luis, Argentina.}%
\email{mzakowi@unsl.edu.ar}%


%\thanks{Research partially supported by ANPCyT (Grant PICT-2007-00865, Foncyt), Universidad Nacional de San Luis (Grant P-317902) and Consejo Nacional de Investigaciones y T\'{e}cnicas (Grant PIP 112-200801-01070), Argentina.}%
%\date{}%
%\dedicatory{}%
%\commby{}%
% ----------------------------------------------------------------

\begin{abstract}
The definition of dual fusion frame presents technical problems
related to the domain of the synthesis operator. The notion commonly
used is the analogous to the canonical dual frame. Here a new
concept of dual is studied in infinite-dimensional separable Hilbert
spaces. It extends the commonly used notion and overcomes these
technical difficulties. We show that with this definition in many
cases dual fusion frames behave similar to dual frames. We present
examples of non-canonical dual fusion frames.
\end{abstract}

\maketitle

% ----------------------------------------------------------------
\section{Introduction}

\textit{Frames} appeared for the first time in the work of Duffin
and Schaeffer in \cite{Duffin-Schaeffer (1952)}. They are systems of
vectors in a separable Hilbert space $\mathcal{H}$ which are
redundant. This means, they allow representations of the elements of
the Hilbert space which are not necessarily unique. This property is
desirable for many situations that appear in applications e.g. in
signal processing when we have presence of noise, since they allow
more flexibility for choosing the adequate representation.  Other
areas where frames are used include coding theory, communication
theory and sampling theory. For a review of frame theory we refer to
\cite{Casazza (2000)} and the books \cite{Christensen (2003),
Daubechies (1992)}. Reference \cite{Casazza-Kutyniok (2013)}
comprises a description of recent developments in theory and
applications of finite frames.

When a huge amount of data has to be processed it is often
advantageous to subdivide a frame system into smaller subsystems and
combine locally data vectors. This gives rise to the concept of
\textit{fusion frames} (or \textit{frames of subspaces})
\cite{Casazza-Kutyniok (2004), Casazza-Kutyniok-Li (2008)} (see also
\cite[Chapter 13]{Casazza-Kutyniok (2013)}), which are a
generalization of frames. They allow decompositions of the elements
of $\mathcal{H}$ into weighted orthogonal projections onto closed
subspaces.

Many concepts of classical frame theory have been generalized to the
setting of fusion frames. Having a proper notion of \emph{dual
fusion frames} permits us to have different reconstruction
strategies. However in the definition appears a problem connected to
the domain of one of the fundamental operators associated to a
fusion frame, called the {\it synthesis operator}. Also, one wants
from a proper definition to yield duality results similar to the known for classical frames.

So far the most frequently used definition is the one that
corresponds to the canonical dual of the classical frames. In this paper the first author proposes a new concept of dual fusion frames. We study properties and provide examples of this new concept
that extends the ``canonical" one and solves the technical problem
mentioned before. Moreover, with this definition we obtain results
which are analogous to those valid for dual frames. In the present
paper we work in infinite-dimensional separable Hilbert spaces
whereas in \cite{Heineken-Morillas (2014)} the new definition is
studied in finite-dimensional Hilbert spaces and it is used to
obtain optimal reconstructions under erasures.

The paper is organized as follows. In Section 2 we review results
about frames and fusion frames. In Section 3 we present the new
definition of dual fusion frames. We then study the relation between
dual fusion frames and the left inverses of the analysis operator.
We present examples of dual fusion frames obtained using this
relation. We show that the ``canonical" dual fusion frame is a
particular case of our duals. We also consider the construction of
dual fusion frames from dual frames. We finally give examples of
non-canonical harmonic dual fusion frames constructed from Gabor
systems.

\section{Preliminaries}

We will now give a brief review about the concepts of frame and
fusion frame

Throughout this paper $\mathcal{H}$ will be a separable Hilbert
space and $I$ a countable index set.

\subsection{Frames and dual frames}

We begin defining the concept of frame.

\begin{defn}\label{D F}
Let $\{f_{i}\}_{i \in I} \subset \mathcal{H}$. Then $\{f_{i}\}_{i
\in I}$ is a \emph{frame}  for $\mathcal{H}$, if there exist
constants $0 < \alpha \leq \beta < \infty$ such that
\begin{equation}\label{E cond f}
\alpha\|f\|^{2} \leq \sum_{i \in I}|\langle f,f_{i}\rangle |^{2}
\leq \beta\|f\|^{2}  \text{ for all $f \in \mathcal{H}$.}
\end{equation}
\end{defn}

If the right inequality in (\ref{E cond f}) is satisfied,
$\{f_{i}\}_{i \in I}$ is a {\it Bessel sequence}. The constants
$\alpha$ and $\beta$ are called {\it frame bounds}. If
$\alpha=\beta,$ we call $\{f_{i}\}_{i \in I}$  an {\it
$\alpha$-tight frame}, and if $\alpha=\beta=1$ it is a {\it Parseval
frame}.

We associate to a Bessel sequence $\{f_{i}\}_{i \in I}$ the {\it
synthesis operator}

\centerline{$T:\ell^2(I)\rightarrow\mathcal{H},$ $T\{c_i\}_{i\in
I}=\sum_{i\in I}c_if_i,$}

\noindent and the {\it analysis operator}

\centerline{$T^*:\mathcal{H}\rightarrow \ell^2(I)$, $T^*f=\{\langle
f,f_i\rangle \}_{i\in I}.$}

For a frame $\{f_{i}\}_{i \in I}$ the operator

\centerline{$S=TT^*$,}

\noindent called {\it frame operator} is positive, selfadjoint and
invertible. Furthermore if $\{f_{i}\}_{i \in I}$ is an
$\alpha$-tight frame, then $S=\alpha I_{\mathcal{H}}$.

The concept of dual frame is defined as follows:

\begin{defn}
Let $\{f_{i}\}_{i \in I}$ be a frame for $\mathcal{H}$ with
synthesis operator $T$. A frame $\{g_{i}\}_{i \in I}$ for
$\mathcal{H}$ with synthesis operator $U$ is a {\em dual frame} of
$\{f_{i}\}_{i \in I}$ if the following reconstruction formula holds
\begin{equation}\label{E def dual marco sintesis}
f=UT^{*}f=\sum_{i \in I}\langle f,f_{i}\rangle g_{i} \text{ , for
all } f \in \mathcal{H}.
\end{equation}
\end{defn}

In particular, $\{S^{-1}f_{i}\}_{i \in I}$ is called the
\emph{canonical dual frame} of $\{f_{i}\}_{i \in I}$.

A {\it Riesz basis} for $\mathcal{H}$ is a frame for $\mathcal{H}$
which is also a basis. Observe that a Riesz basis has a unique dual,
the canonical one.

\subsection{Fusion frames}

Let $\{W_{i}\}_{i \in I}$ be a family of closed subspaces in
$\mathcal{H}$, and let $\{w_{i}\}_{i \in I}$ be a family of weights,
i.e., $w_{i}
> 0$ for all $i \in I$. We will denote $\{W_{i}\}_{i \in I}$ with $\mathcal{W}$, $\{w_{i}\}_{i
\in I}$ with $w$ and $\{(W_i,w_{i})\}_{i \in I}$ with
$(\mathcal{W},w)$. If $T \in L(\mathcal{H},\mathcal{K})$ we will
write $(T\mathcal{W},w)$ for $\{(TW_i, w_{i})\}_{i \in I}.$

We consider the Hilbert space

\centerline{$\mathcal{K}_{\mathcal{W}}:=\bigoplus_{i \in I}W_{i} =
\{\{f_{i}\}_{i \in I}:f_{i} \in W_{i} \text{ and } \{\|f_{i}\|\}_{i
\in I} \in \ell^{2}(I)\}$}

\noindent with inner product $\langle\{f_{i}\}_{i \in
I},\{g_{i}\}_{i \in I}\rangle=\sum_{i \in I}\langle f_{i},
g_{i}\rangle.$

For $V$ a closed subspace of $\mathcal{H}$, $\pi_{V}$ is the
orthogonal projection onto $V$.

\begin{defn}\label{D fusion frame}
We say that $(\mathcal{W},w)$ is a {\it fusion frame} for
$\mathcal{H}$, if there exist constants $0 < \alpha \leq \beta <
\infty$ such that
\begin{equation}\label{E cond ff}
\alpha\|f\|^{2} \leq \sum_{i \in I}w_{i}^{2}\|\pi_{W_{i}}(f)\|^{2}
\leq \beta\|f\|^{2}  \text{ for all $f \in \mathcal{H}$.}
\end{equation}
\end{defn}

We call $\alpha$ and $\beta$ the \textit{fusion frame bounds}. The
family $(\mathcal{W},w)$ is called an $\alpha$-\textit{tight fusion
frame}, if in (\ref{E cond ff}) the constants $\alpha$ and $\beta$
can be chosen so that $\alpha = \beta$, and a \textit{Parseval
fusion frame} provided that $\alpha = \beta = 1.$ If
$(\mathcal{W},w)$ possesses an upper fusion frame bound, but not
necessarily a lower bound, we call it a \textit{Bessel fusion
sequence}  with Bessel fusion bound $\beta.$ If $w_{i}=c$ for all
$i\in I,$ the collection $(\mathcal{W},w)$ is called
$c$-\textit{uniform}. In this case we write $(\mathcal{W},c)$. We
refer to a fusion frame $(\mathcal{W},1)$ as an \textit{orthonormal
fusion basis} if {$\mathcal{H}$ is the orthogonal sum of the
subspaces $W_{i}.$}

We associate to a Bessel fusion sequence $(\mathcal{W},w)$ the
following bounded operators:

\centerline{$T_{\mathcal{W},w} : \mathcal{K}_{\mathcal{W}}
\rightarrow \mathcal{H},$ $T_{\mathcal{W},w}\{f_{i}\}_{i \in
I}=\sum_{i \in
  I}w_{i}f_{i},$}

\noindent called the \textit{synthesis operator} and

\centerline{$T_{\mathcal{W},w}^{*} : \mathcal{H} \rightarrow
\mathcal{K}_{\mathcal{W}},$
$T_{\mathcal{W},w}^{*}f=\{w_{i}\pi_{W_{i}}(f)\}_{i \in I}$,}

\noindent named the \textit{analysis
  operator}.

If $(\mathcal{W},w)$ is a \textit{fusion frame} we have the {\it
fusion frame operator}

\centerline{ $
S_{\mathcal{W},w}=T_{\mathcal{W},w}T_{\mathcal{W},w}^{*},$}

\noindent which is positive, selfadjoint and invertible. If
$(\mathcal{W},w)$ is an $\alpha$-tight fusion frame, then
$S_{\mathcal{W},w}=\alpha I_{\mathcal{H}}.$

As for frames, $(\mathcal{W},w)$ is a Bessel fusion sequence for
$\mathcal{H}$ if and only if $T_{\mathcal{W},w}$ is a well defined
bounded linear operator. A Bessel fusion sequence $(\mathcal{W},w)$
is a fusion frame for $\mathcal{H}$ if and only if
$T_{\mathcal{W},w}$ is onto.

\begin{rem}
For $w \in \ell^{2}(I)$ it is easy to see that $(\mathcal{W},w)$ is
a Bessel fusion sequence for $\mathcal{H}$. Suppose that
$T_{\mathcal{W},w}$ is well defined. If $T_{\mathcal{W},w}$ is
bounded then $w \in \ell^{\infty}(I)$. In view of this, in the
sequel we suppose that each family of weights is in
$\ell^{\infty}(I)$.
\end{rem}

%%%%%%%%%%%%%%%%%%%%%%%%%%%%%%%%%%%%%%%%%%%%%%%%%%%5

\section{Dual fusion frames}

Assume that $(\mathcal{W},w)$ is a fusion frame for $\mathcal{H}$.
In \cite[Definition 3.19]{Casazza-Kutyniok (2004)} the fusion frame
$(S_{\mathcal{W},w}^{-1}\mathcal{W},w)$ is called the \textit{dual
fusion frame} of $(\mathcal{W},w).$ This family looks as the
analogous to the canonical dual frame in the classical frame theory
and
\begin{equation}\label{E reconst dual fusion frame canonico}
f=S^{-1}_{\mathcal{W},w}S_{\mathcal{W},w}f=\sum_{j \in
I}w_{j}^{2}S_{\mathcal{W},w}^{-1}\pi_{W_{j}}(f) \text{ , for all } f
\in \mathcal{H}.
\end{equation}

If we like (\ref{E reconst dual fusion frame canonico}) expressed in
terms of operators as in (\ref{E def dual marco sintesis}), we find
the following obstacle. We have $R(T^{*}_{\mathcal{W},w}) \subseteq
\mathcal{K}_{\mathcal{W}}$ and
$D(T_{S^{-1}_{\mathcal{W},w}W,w})=\bigoplus_{i \in I}
S^{-1}_{\mathcal{W},w}W_{i}$, so
$T_{S_{\mathcal{W},w}^{-1}W,w}T^{*}_{\mathcal{W},w}$ is generally
not defined. This difficulty with the domain disappears with the
following definition of dual fusion frame.
\begin{defn}\label{D fusion frame dual}
Assume that $(\mathcal{W},w)$ and $(\mathcal{V},v)$ are fusion
frames for $\mathcal{H}$. If there exists $Q \in
L(\mathcal{K}_{\mathcal{W}},\mathcal{K}_{\mathcal{V}})$ such that
\begin{equation}\label{E TvQTw*=I}
T_{\mathcal{V},v}QT^{*}_{\mathcal{W},w}=I_{\mathcal{H}},
\end{equation}
then $(\mathcal{V},v)$ is a {\em dual fusion frame} of
$(\mathcal{W},w)$
\end{defn}

Sometimes we will write  {\it $Q$-dual fusion frame} in case we need
the operator $Q$ to be explicit.

The following lemma collects some properties of dual fusion frames
that are analogous to corresponding ones for dual frames with
similar proofs (see, e. g., \cite[Lemma 5.6.2]{Christensen (2003)}).

\begin{lem}\label{dualgral}
Let $(\mathcal{W},w)$ and $(\mathcal{V},v)$ be Bessel fusion
sequences for $\mathcal{H}$, and let $Q \in
L(\mathcal{K}_{\mathcal{W}},\mathcal{K}_{\mathcal{V}})$. Then the
following conditions are equivalent:
\begin{enumerate}
  \item $T_{\mathcal{V},v}QT^{*}_{\mathcal{W},w}=I_{\mathcal{H}}$.
  \item $T_{\mathcal{W},w}Q^{*}T^{*}_{\mathcal{V},v}=I_{\mathcal{H}}$.
  \item $T^{*}_{\mathcal{W},w}$ is injective, $T_{\mathcal{V},v}Q$ is surjective and
  $(T^{*}_{\mathcal{W},w}T_{\mathcal{V},v}Q)^{2}=T^{*}_{\mathcal{W},w}T_{\mathcal{V},v}Q$.
  \item $T^{*}_{\mathcal{V},v}$ is injective, $T_{\mathcal{W},w}Q^{*}$ is surjective and
  $(T^{*}_{\mathcal{V},v}T_{\mathcal{W},w}Q^{*})^{2}=T^{*}_{\mathcal{V},v}T_{\mathcal{W},w}Q^{*}$.
  \item $\langle f,g \rangle=\langle QT^{*}_{\mathcal{W},w}f,T_{\mathcal{V},v}^{*}g\rangle=\langle Q^{*}T^{*}_{\mathcal{V},v}f,T_{\mathcal{W},w}^{*}g\rangle$
  for all $f, g \in \mathcal{H}$.
\end{enumerate}
In case any of these equivalent conditions are satisfied,
$(\mathcal{W},w)$ and $(\mathcal{V},v)$ are fusion frames for
$\mathcal{H}$, $(\mathcal{V},v)$ is a $Q$-dual fusion frame of
$(\mathcal{W},w)$, and $(\mathcal{W},w)$ is a $Q^{*}$-dual fusion
frame of $(\mathcal{V},v)$.
\end{lem}

Note that the equivalence of conditions (1) and (2) implies that the
roles of $(\mathcal{W},w)$ and $(\mathcal{V},v)$ can be interchanged
in the definition of dual fusion frame. Conditions (3) and (4) say
that the $Q$-mixed Gram operator is a projection if
$(\mathcal{W},w)$ and $(\mathcal{V},v)$ are dual fusion frames.
Finally, condition (5) expresses the inner product of two elements
of $\mathcal{H}$ in terms of a $Q$-inner product of their images
under the analysis operators of $(\mathcal{W},w)$ and
$(\mathcal{V},v).$

\subsection{Dual fusion frames obtained from left inverses of the analysis
operator}

In classical frame theory, dual frames are related to the left
inverses of the analysis operator. An analogous result is valid for
the following special type of dual fusion frame.
\begin{defn}Let $p_{i}: \mathcal{K}_{\mathcal{W}} \rightarrow \mathcal{K}_{\mathcal{W}}$,
$p_{i}\{f_{j}\}_{j \in I}=\{\delta_{i,j}f_{j}\}_{j \in I}$. If $Q$
in Definition~\ref{D fusion frame dual} satisfies
$$Qp_{i}\mathcal{K}_{\mathcal{W}}=p_{i}\mathcal{K}_{\mathcal{V}},$$ we say that $Q$ is {\em component preserving} and
$(\mathcal{V},v)$ is a \emph{component preserving dual fusion frame}
of $(\mathcal{W},w)$.
\end{defn}

\noindent \textbf{Notation.} We denote the set of bounded left
inverses of $T^{*}_{\mathcal{W},w}$ with
$\mathfrak{L}_{T^{*}_{\mathcal{W},w}}$.

\smallskip

The next two lemmata link component preserving dual fusion frames of
$(\mathcal{W},w)$ with the left inverses of the analysis operator
$T^{*}_{\mathcal{W},w}$. They are analogous to a corresponding
result for dual frames (see, e. g., \cite[Lemma 5.6.3.]{Christensen
(2003)}).

\begin{lem}\label{L V,v dual fusion frame entonces Vi=ApiWj}
Let $(\mathcal{W},w)$ be a fusion frame for $\mathcal{H}$. If
$(\mathcal{V},v)$ is a component preserving dual fusion frame of
$(\mathcal{W},w)$ then $V_{i}=Ap_{i}\mathcal{K}_{\mathcal{W}}$, for
each $i \in I$, where $A \in \mathfrak{L}_{T^{*}_{\mathcal{W},w}}$.
\end{lem}
\begin{proof}
Let $Q\in L(\mathcal{K}_{\mathcal{W}},\mathcal{K}_{\mathcal{V}})$ be
component preserving such that
$T_{\mathcal{V},v}QT^{*}_{\mathcal{W},w}=I_{\mathcal{H}}$. Let
$A=T_{\mathcal{V},v}Q$. Clearly, $A \in
\mathfrak{L}_{T^{*}_{\mathcal{W},w}}$. Since $Q$ is component
preserving,
$Ap_{i}\mathcal{K}_{\mathcal{W}}=T_{\mathcal{V},v}Qp_{i}\mathcal{K}_{\mathcal{W}}=T_{\mathcal{V},v}p_{i}\mathcal{K}_{\mathcal{V}}=V_{i}.$\end{proof}

Fusion frames behave differently under operators than classical
frames do (see e.g. \cite{Casazza-Kutyniok-Li (2008), Ruiz-Stojanoff
(2008)}). A well known result in classical frame theory is that if
$A \in L(\mathcal{K}, \mathcal{H})$ is surjective and
$\set{e_{i}}_{i \in I}$ is an orthonormal basis for $\mathcal{K}$,
then $\set{Ae_{i}}_{i \in I}$ is a frame for $\mathcal{H}$. In the
context of fusion frames the situation is different. Given an
orthonormal fusion basis $(\mathcal{W},1)$  of $\mathcal{K},$ a
surjective $A \in L(\mathcal{K}, \mathcal{H})$ and a family of
weights $v,$ the collection $(A\mathcal{W},v)$ could even not be a
Bessel sequence for $\mathcal{H}$. This fact does not allow to have
the complete converse of the previous result in an
infinite-dimensional separable Hilbert space. However, a reciprocal
is valid in the following sense:

\begin{lem}\label{L Vi=ApiWj entonces V,v dual fusion frame}
Let $(\mathcal{W},w)$ be a fusion frame for $\mathcal{H}$, $A \in
\mathfrak{L}_{T^{*}_{\mathcal{W},w}}$ and
$V_{i}=Ap_{i}\mathcal{K}_{\mathcal{W}}$ for each $i \in I$. If
$(\mathcal{V},v)$ is a Bessel fusion sequence and

\centerline{$Q_{A,v}: \mathcal{K}_{\mathcal{W}} \rightarrow
\mathcal{K}_{\mathcal{V}}, ~~Q_{A,v}\{f_{j}\}_{j \in
I}=\{\frac{1}{v_{i}}Ap_{i}\{f_{j}\}_{j \in I}\}_{i \in I}$}

\noindent is a well defined bounded operator, then $(\mathcal{V},v)$
is a $Q_{A,v}$-component preserving dual fusion frame of
$(\mathcal{W},w)$.
\end{lem}

\begin{proof}
Let $\{f_{j}\}_{j \in I} \in \mathcal{W}$. We have
$Ap_{i}p_{i_{0}}\{f_{j}\}_{j \in I}
=\delta_{i,i_{0}}Ap_{i_{0}}\{f_{j}\}_{j \in I}\in
Ap_{i_{0}}\mathcal{K}_{\mathcal{W}}\\=V_{i_{0}}.$ Therefore,
$Q_{A,v}p_{i_{0}}\{f_{j}\}_{j \in
I}=\{\delta_{i,i_{0}}\frac{1}{v_{i_{0}}}Ap_{i_{0}}\{f_{j}\}_{j \in
I}\}_{i \in I} \in p_{i_{0}}\mathcal{K}_{\mathcal{V}}.$
Consequently, $Q_{A,v}p_{i_{0}}\mathcal{K}_{\mathcal{W}} \subseteq
p_{i_{0}}\mathcal{K}_{\mathcal{V}}.$ For the other inclusion, let
$\{g_{i}\}_{i \in I} \in \mathcal{K}_{\mathcal{V}}$. Then
$g_{i}=Ap_{i}\{f_{j}^{i}\}_{j \in I}$ with $\{f_{j}^{i}\}_{j \in I}
\in \mathcal{K}_{\mathcal{W}}$ for each $i \in I,$ and thus
$p_{i_{0}}\{g_{i}\}_{i \in
I}=\{\delta_{i,i_{0}}Ap_{i_{0}}\{f_{j}^{i_{0}}\}_{j \in I}\}_{i \in
I} = \{\delta_{i,i_{0}}A\{\delta_{j,i_{0}}f_{i_{0}}^{i_{0}}\}_{j \in
I}\}_{i \in I}.$ Since
$v_{i_{0}}\{\delta_{j,i_{0}}f_{i_{0}}^{i_{0}}\}_{j \in I} \in
p_{i_{0}}\mathcal{K}_{\mathcal{W}}$ and
$Ap_{i}v_{i_{0}}\{\delta_{j,i_{0}}f_{i_{0}}^{i_{0}}\}_{j \in
I}=\delta_{i,i_{0}}v_{i_{0}}A\{\delta_{j,i_{0}}f_{i_{0}}^{i_{0}}\}_{j
\in I}$, then $p_{i_{0}}\{g_{i}\}_{i \in
I}=Q_{A,v}v_{i_{0}}\{\delta_{j,i_{0}}f_{i_{0}}^{i_{0}}\}_{j \in I}.$
So, $p_{i_{0}}\mathcal{K}_{\mathcal{V}} \subseteq
Q_{A,v}p_{i_{0}}\mathcal{K}_{\mathcal{W}}.$ This shows that
$Q_{A,v}$ is component preserving.

Since $(\mathcal{V},v)$ is a Bessel fusion sequence,
$T_{\mathcal{V},v}$ is a well defined bounded linear operator. If
$\{f_{j}\}_{j \in I} \in \mathcal{K}_{\mathcal{W}},$ then

\centerline{$T_{\mathcal{V},v}Q_{A,v}\{f_{j}\}_{j \in I}= \sum_{i
\in I}v_{i}\frac{1}{v_{i}}Ap_{i}\{f_j\}_{j \in I}= A\sum_{i \in
I}p_{i}\{f_j\}_{j \in I}= A\{f_j\}_{j \in I}.$} \noindent  Hence
$T_{\mathcal{V},v}Q_{A,v}=A \in
\mathfrak{L}_{T^{*}_{\mathcal{W},w}}$. So $(\mathcal{V},v)$ is a
$Q_{A,v}$-component preserving dual fusion frame of
$(\mathcal{W},w)$.
\end{proof}
\begin{rem}\label{R V Bessel Q acotada}
Let $A$, $(\mathcal{V},v)$ and $Q_{A,v}$ as in Lemma~\ref{L Vi=ApiWj
entonces V,v dual fusion frame}. We can give the following
sufficient conditions for $(\mathcal{V},v)$ being a Bessel fusion
sequence and for $Q_{A,v}$ being a well defined bounded operator:

(1) Let $\gamma(A)$ be the reduced minimum modulus of $A$, i.e.,
$\gamma(A)=\text{inf}\{\|Ax\|:\|x\|=1, x \in N(A)^{\perp}\}.$ Assume
$\gamma(Ap_{i}) > 0$ and there exists $\delta > 0$ such that $\delta
\leq v_{i}^{-2}\gamma(Ap_{i})^{2}$ for all $i \in I.$ Since
$\{(p_{i}\mathcal{K}_{\mathcal{W}},1)\}_{i \in I}$ is an orthonormal
fusion basis for $\mathcal{K}_{\mathcal{W}}$, by \cite[Theorem
3.6]{Ruiz-Stojanoff (2008)} $(\mathcal{V},v)$ is a Bessel fusion
sequence for $\mathcal{H}$ with upper bound
$\frac{\|A\|^{2}}{\delta}$.

(2) If $v_{i} > \delta > 0$ for each $i \in I$, then $Q_{A,v}$ is a
well defined bounded operator with $\|Q_{A,v}\| \leq
\frac{\|A\|}{\delta}$. To see this, note that if $\{f_{j}\}_{j \in
I} \in \mathcal{K}_{\mathcal{W}}$, then

\centerline{$\sum_{i \in I}\|\frac{1}{v_{i}}Ap_{i}\{f_{j}\}_{j \in
I}\|\leq \frac{\|A\|^{2}}{\delta^{2}}\sum_{i \in
I}\|p_{i}\{f_{j}\}_{j \in
I}\|^{2}=\frac{\|A\|^{2}}{\delta^{2}}\|\{f_{i}\}_{i \in I}\|^{2} .$}
\end{rem}
\begin{example}\label{Ej dual canonico}
Any reconstruction formula of the form $f= A T^*_{\mathcal{W},w} f $
involves a left inverse $A$ of $T^*_{\mathcal{W},w}$, and then, in
view of Lemma~\ref{L Vi=ApiWj entonces V,v dual fusion frame}, it
could be expressed in terms of a $Q_{A,v}$-component preserving dual
of $(\mathcal{W},w)$. The present example is based on this
observation.

Since
$(S_{\mathcal{W},w}^{-1}T_{\mathcal{W},w})T^{*}_{\mathcal{W},w}=I_{\mathcal{H}}$,
then $A=S_{\mathcal{W},w}^{-1}T_{\mathcal{W},w} \in
\mathfrak{L}_{T^{*}_{\mathcal{W},w}}$. We have

\centerline{$A p_i
\mathcal{K}_{\mathcal{W}}=S_{\mathcal{W},w}^{-1}W_i, ~~\forall i\in
I$}

\noindent and that

\centerline{$Q_{A,w}  : \mathcal{K}_{\mathcal{W}} \rightarrow
\bigoplus_{i \in I} S_{\mathcal{W},w}^{-1}W_{i}, ~~Q_{A,w}
\{f_{j}\}_{j \in I}= \{S_{\mathcal{W},w}^{-1}f_{i}\}_{i \in I}$}

\noindent is a well defined bounded operator with $\|Q_{A,w}\| \leq
\|S_{\mathcal{W},w}^{-1}\|$. So, by Lemma~\ref{L Vi=ApiWj entonces
V,v dual fusion frame}, the dual fusion frame introduced in
\cite{Casazza-Kutyniok (2004)},
$(S_{\mathcal{W},w}^{-1}\mathcal{W},w),$ is a component preserving
dual fusion frame of $(\mathcal{W},w)$.

Note that in this example $Q_{A,w}$ is bounded without any
additional restriction on the weights $w_{i}$.
\end{example}
In the sequel we refer to $(S_{\mathcal{W},w}^{-1}\mathcal{W},w)$ as
the canonical dual and to

\centerline{$Q_{S_{\mathcal{W},w}^{-1}T_{\mathcal{W},w},w}^{*}T_{S_{\mathcal{W},w}^{-1}\mathcal{W},w}^{*}f=T_{\mathcal{W},w}^{*}S_{\mathcal{W},w}^{-1}f
\,\in \mathcal{K}_{\mathcal{W}}$}

\noindent as the \textit{fusion frame coefficients} of $f \in
\mathcal{H}$.

The next lemma is about the minimality of the fusion frame
coefficients and has its analogous in classical frame theory with a
similar proof (see e.g. \cite[Lemma 5.4.2]{Christensen (2003)}).

\begin{lem}
Let $(\mathcal{W},w)$ be a fusion frame for $\mathcal{H}$ and $f \in
\mathcal{H}$. For all $\{f_{j}\}_{j \in I} \in
\mathcal{K}_{\mathcal{W}}$ satisfying $T_{\mathcal{W},w}\{f_{j}\}_{j
\in I}=f$ we have $\|T_{\mathcal{W},w}^{*}S_{\mathcal{W},w}^{-1}f\|
\leq \|\{f_{j}\}_{j \in I}\|$.
\end{lem}

\begin{example}\label{ejemploA}
Assume $\{e_{j}\}_{j=1}^{\infty}$ is an orthonormal basis of
$\mathcal{H}$. Fix $N \in \mathbb{N}$ and define
$W_{j}=\overline{\text{span}}\{e_{1}, \ldots, e_{j}\}$ for $1 \leq j
\leq N$ and $W_{j}=\overline{\text{span}}\{e_{j-N+1},\ldots,
e_{j}\}$ for $j > N$. Then $(\mathcal{W},1)$ is a $1-$uniform
$N$-tight fusion frame for $\mathcal{H}$.

To simplify the exposition we consider in the sequel $N=3$ and
$e_{-1}=e_{0}=0$. In this case
$\mathcal{K}_{\mathcal{W}}=\{\{\sum_{k=0}^{2}c_{j,j-k}e_{j-k}\}_{j=1}^{\infty}:\sum_{j=1}^{\infty}\sum_{k=0}^{2}|c_{j,j-k}|^{2}<\infty\}$
and $T_{W,1}^{*}:\mathcal{H}\rightarrow \mathcal{K}_{\mathcal{W}}
\text{ ,
}T_{W,1}^{*}\sum_{j=1}^{\infty}c_{j}e_{j}=\{\sum_{k=0}^{2}c_{j-k}e_{j-k}\}_{j=1}^{\infty}.$

Consider $A:\mathcal{K}_{\mathcal{W}}\rightarrow \mathcal{H}$,
$A\{\sum_{k=0}^{2}c_{j,j-k}e_{j-k}\}_{j=1}^{\infty}=\sum_{j=1}^{\infty}c_{j,j}e_{j}.$
Then $A \in \mathfrak{L}_{T^{*}_{\mathcal{W},1}}$,
$V_{i}=Ap_{i}\mathcal{K}_{\mathcal{W}}=\overline{\text{span}}\{e_{i}\},$
$(\mathcal{V},1)$ is a $1-$uniform orthonormal fusion basis for
$\mathcal{H}$ and $Q_{A,1}: \mathcal{K}_{\mathcal{W}} \rightarrow
\mathcal{K}_{\mathcal{V}}, ~
Q_{A,1}\{\sum_{k=0}^{2}c_{j,j-k}e_{j-k}\}_{j=1}^{\infty}=\{c_{i,i}e_{i}\}_{i=1}^{\infty}$
is bounded.

Note that since $ (\mathcal{V},1)$ is an orthonormal fusion basis
for $\mathcal{H},$ it coincides with its unique component preserving
dual fusion frame. On the other hand, by Lemma~\ref{dualgral},
$(\mathcal{W},1)$ gives an example of dual of $(\mathcal{V},1)$,
which is not component preserving.
\end{example}

\begin{rem}\label{noBFS}
The following shows that in Lemma~\ref{L Vi=ApiWj entonces V,v dual
fusion frame}, the hypotheses $(\mathcal{V},v)$ to be a Bessel
fusion sequence and $Q_{A,v}$ to be bounded, cannot be avoided:

(1) Let $\{e_n\}_{n\in\N}$ be an orthonormal basis of a Hilbert
space $\mathcal{H}$ and consider
$W_1=\overline{\text{span}}\{e_k:k\geq 2\}=\{e_1\}^{\bot}$ and
$W_k=\overline{\text{span}}\{e_1,e_k\}$ for $k\geq 2.$

It is proved in \cite[Example 7.5]{Ruiz-Stojanoff (2008)} that if
$(\mathcal{W},w)$ is a Bessel fusion sequence, then $w\in
\ell^{2}(\N).$  Moreover, the frame operator $S_{\mathcal{W},w}$ is
diagonal with respect to $\{e_n\}_{n\in\N}$ and so it is also
$S_{\mathcal{W},w}^{-1}.$

Now, in particular, this implies that
$S_{\mathcal{W},w}^{-1}W_k=W_k$ for all $k\in \N.$ So if
$A=S_{\mathcal{W},w}^{-1}T_{\mathcal{W},w}$ then $V_k:=A p_k
\mathcal{K}_{\mathcal{W}}= S_{\mathcal{W},w}^{-1}W_k.$ Thus, if $v
\in \ell^{\infty}(\N) \setminus \ell^{2}(\N),$ then
$(\mathcal{V},v)$ is not a Bessel fusion sequence.

(2) There exist weights $v_{i}$ so that
$Q_{S_{\mathcal{W},w}^{-1}T_{\mathcal{W},w},v}$ is unbounded.
Specifically, let $v_{i} \leq \frac{w_{i}}{i\|S_{\mathcal{W},w}\|}$.
If $\{f_{j}^{(i)}\}_{j\in I} \in p_{i}\mathcal{K}_{\mathcal{W}}$,
then
\begin{align*}
\|Q_{S_{\mathcal{W},w}^{-1}T_{\mathcal{W},w},v}\{f_{j}^{(i)}\}_{j\in
I}\|&=\frac{w_{i}}{v_{i}}\|S_{\mathcal{W},w}^{-1}f_{i}^{(i)}\|
\geq \frac{w_{i}}{v_{i}\|S_{\mathcal{W},w}\|}\|f_{i}^{(i)}\|\\
&= \frac{w_{i}}{v_{i}\|S_{\mathcal{W},w}\|} \|\{f_{j}^{(i)}\}_{j\in
I}\|\geq i \|\{f_{j}^{(i)}\}_{j\in I}\|.
\end{align*}
\end{rem}

Let $(\mathcal{W},w)$ be a fusion frame for $\mathcal{H}.$ It can be
seen as in \cite[Lemma~5.6.4]{Christensen (2003)}, that the bounded
left inverses of $T_{\mathcal{W},w}^{*}$ are the operators $A$ of
the form
$$A=S_{\mathcal{W},w}^{-1}T_{\mathcal{W},w}+R\paren{I_{\mathcal{K}_{\mathcal{W}}}-T_{\mathcal{W},w}^{*}S_{\mathcal{W},w}^{-1}T_{\mathcal{W},w}},$$
\noindent where $R\in L(\mathcal{K}_{\mathcal{W}},\mathcal{H})$.
This fact, Lemma~\ref{L V,v dual fusion frame entonces Vi=ApiWj} and
Remark~\ref{R V Bessel Q acotada} yield the following description
for component preserving dual fusion frames:

\begin{thm}
Let $(\mathcal{W},w)$  be a fusion frame for $\mathcal{H}$. Suppose
that $v_{i} > \delta > 0$ for each $i \in I$. Then the component
preserving dual fusion frames of $(\mathcal{W},w)$ are the Bessel
fusion sequences $(\mathcal{V},v)$ where

\centerline{$V_{i}=\left[S_{\mathcal{W},w}^{-1}T_{\mathcal{W},w}+R\paren{I_{\mathcal{K}_{\mathcal{W}}}-T_{\mathcal{W},w}^{*}S_{\mathcal{W},w}^{-1}T_{\mathcal{W},w}}\right]\paren{p_{i}\mathcal{K}_{\mathcal{W}}}$}

\noindent and $R\in L(\mathcal{K}_{\mathcal{W}},\mathcal{H}).$
\end{thm}

\subsection{Dual fusion frames obtained from dual frames}

The following theorem provides a method to obtain dual fusion
frames.

First recall that if $\{f_{i}\}_{i \in I}\subset \mathcal{H}$ is a
Bessel sequence with bound $\beta$, then
\begin{equation}\label{E cond B1}
\|\sum_{i \in I}c_{i}f_{i}\|^{2} \leq \beta\|c\|^{2} \text{ for all
$c \in \ell^{2}(I)$.}
\end{equation}

\begin{thm}\label{T fusion frame system dual}
For each $i \in I$, let $w_{i} > 0$, $v_{i} > 0$, and let $W_{i}$
and $V_{i}$ be closed subspaces of $\mathcal{H}$. Let
$\{\fil\}_{l\in L_i}$ be a frame for $W_{i}$ and
$\{\tilde{f}_i^l\}_{l\in L_i}$  be a frame for $V_{i}$, with frame
bounds $\alpha_i,$ $\beta_i,$ $\tilde{\alpha}_i$ and
$\tilde{\beta}_i$, respectively. Suppose that $0<\alpha= \inf_{i\in
I}\alpha_i\leq \beta=\sup_{i\in I}\beta_i<\infty$ and
$0<\tilde{\alpha}=\inf_{i\in I}\tilde{\alpha}_i\leq
\tilde{\beta}=\sup_{i\in I}\tilde{\beta}_i<\infty$. Let
$Q:\mathcal{K}_{\mathcal{W}}\To \mathcal{K}_{\mathcal{V}}$,
$Q\{h_i\}_{i\in I}:=\{\sum_{l\in L_i}\langle h_i,f_i^l\rangle
\tilde{f}_i^l\}_{i\in I}$. The following conditions are equivalent.
  \begin{enumerate}
    \item $\{w_i\fil\}_{i\in I, l\in L_i}$ and $\{v_i\fmonio\}_{i\in I, l\in L_i}$ are dual frames in
    $\mathcal{H}$.
    \item $(\mathcal{V},v)$ is a $Q$-dual fusion frame of $(\mathcal{W},w)$.
  \end{enumerate}
\end{thm}

\begin{proof}
By \cite[Theorem 3.2]{Casazza-Kutyniok (2004)} it only remains to
see the duality. If $\{h_i\}_{i\in I} \in
\mathcal{K}_{\mathcal{W}}$, by (\ref{E cond B1}) and (\ref{E cond
f}),

\centerline{$\sum_{i\in I}\|\sum_{l\in L_i}\langle h_i,f_i^l\rangle
\tilde{f}_i^l\|^{2}\leq\sum_{i\in
I}\tilde{\beta}_{i}^2\beta_{i}^2\|h_i\|^2
\leq\tilde{\beta}^2\beta^2\sum_{i\in I}\|h_i\|^2 < \infty.$}

\noindent  So $Q$ is a well defined bounded operator.

Using that $\langle \pi_{W_i}(f),f_i^l\rangle=\langle
f,f_i^l\rangle$, we obtain

\centerline{$T_{\mathcal{V},v}QT^*_{\mathcal{W},w}(f)=\sum_{i\in
I}\sum_{l\in L_i}\langle f,w_i f_i^l\rangle v_i\tilde{f}^l_i.$}

\noindent Finally, the last term is equal to $f$ for all $f\in
\mathcal{H}$ if and only if $\{w_if^l_i\}_{i\in I, l\in L_i}$ and
$\{v_i\tilde{f}^l_i\}_{i\in I, l\in L_i}$ are dual frames in
$\mathcal{H}$.
\end{proof}
In the following example we present harmonic fusion frames with
non-canonical dual fusion frames, which are also harmonic.
\begin{example}
Let $a \in \mathbb{R}$, $g\in L^2(\Real)$, $E_{a}g(x):=e^{2\pi i a
x}g(x)$ and $T_{a}g(x):=g(x-a)$. Suppose that the so-called Gabor
system $\{E_{am}T_{n}g\}_{m,n\in \Z}$ is a Parseval frame. Fix $N
\in \mathbb{N}$ and define

\centerline{$W_{i}=\overline{\text{span}}\{E_{a(Nm+i)}T_{n}g\}_{m,n\in
\Z}$, $~~0 \leq i \leq N-1$.}

\noindent We have $W_{0}=E_{a}W_{N-1} ~\text{ and }~
W_{i+1}=E_{a}W_{i} ~\text{ for }~ 0 \leq i \leq N-2.$ The family
$(\mathcal{W},1)$ is the finite harmonic fusion frame considered in
\cite[Example 6.4]{Casazza-Kutyniok (2004)}.

Let now $d \in \mathbb{C}$ and $N \in \mathbb{N}$ be such that
$\frac{1}{\sqrt{2}} <|d|<1$, $|d|^{2}N \in \mathbb{N}$ and $|d|^{2}N
>1$. Let $c_{i} \in \mathbb{C}$ for $i=1, \ldots, N-|d|^{2}N,$ with
some $c_i\neq 0,$ such that $\sum_{i=1}^{N-|d|^{2}N}c_{i}=0.$ Set

\centerline{$g=d\chi_{[0,1)}\,\,\text{ and
}\,\,\,h=d\chi_{[0,1)}+\sum_{i=1}^{N-|d|^{2}N}c_{i}\chi_{[1+\frac{i-1}{|d|^{2}N},1+\frac{i}{|d|^{2}N})}.$}

\noindent Let
$W_{i}=\overline{\text{span}}\{E_{|d|^2(Nm+i)}T_{n}g\}_{m,n\in \Z}$
and $V_{i}=\overline{\text{span}}\{E_{|d|^2(Nm+i)}T_{n}h\}_{m,n\in
\Z}$, for $0 \leq i \leq N-1$. We are going to show that the finite
harmonic fusion frame $(\mathcal{V},1)$ is a dual fusion frame of
$(\mathcal{W},1)$, and it is not the canonical dual.

\smallskip

We have $\sum_{n \in \mathbb{Z}}|g(x-n)|=|d|^{2}$ and $\sum_{n \in
\mathbb{Z}}g(x-n)\overline{g(x-n-\frac{k}{|d|^{2}})}=0$ a.e.. By
\cite[Theorems 9.5.2 (ii) and 8.3.1 (ii)]{Christensen (2003)},
$\{E_{|d|^{2}m}T_{n}g\}_{m,n\in \Z}$ is a Parseval Gabor frame but
not a Riesz basis for $L^2(\Real)$.

\smallskip

Since $\{E_{|d|^{2}mN}T_{n}g\}_{m, n \in \mathbb{Z}} \subseteq
\{E_{|d|^{2}m}T_{n}g\}_{m, n \in \mathbb{Z}}$ and
$\{E_{|d|^{2}m}T_{n}g\}_{m, n \in \mathbb{Z}}$ is a Bessel sequence
for $L^2(\Real)$, then $\{E_{|d|^{2}mN}T_{n}g\}_{m, n \in
\mathbb{Z}}$ is a Bessel sequence for $W_{0}$. Moreover, it is an
orthogonal system with elements of equal norm $|d|$. Consequently,
the associated well defined frame operator is
$|d|^{2}I_{L^2(\Real)}$ and $\{E_{|d|^{2}mN}T_{n}g\}_{m,n\in \Z}$ is
an orthogonal basis for $W_{0}$. Thus, since $E_{|d|^{2}Ni}$ is a
unitary operator, $\{E_{|d|^{2}N(mN+i)}T_{n}g\}_{m,n\in \Z}$ is a
$|d|^{2}$-tight frame for $W_{i}$, $0 \leq i \leq N-1$.

\smallskip

By \cite[Theorem 8.3.1 (i)]{Christensen (2003)}, $W_{0} \neq
L^{2}(\mathbb{R})$. Hence, since $E_{|d|^{2}Ni}$ is a unitary
operator, $W_{i} \neq L^{2}(\mathbb{R})$ for $0 \leq i \leq N-1$.

\smallskip

By \cite[Proposition 5.3.5]{Christensen (2003)}, using that
$\{E_{|d|^{2}(Nm+i)}T_{n}g\}_{m,n\in \Z}$ is a $|d|^{2}$-tight frame
for $W_{i}$, $0 \leq i \leq N-1$, and
$\{E_{|d|^{2}m}T_{n}g\}_{m,n\in \Z}$ is a Parseval Gabor frame for
$L^2(\Real)$, we obtain $S_{W,1}=|d|^{-2}I_{L^2(\Real)}.$ Thus
$(\mathcal{W},1)$ is a finite harmonic $|d|^{-2}$-tight fusion
frame, so it coincides with its canonical dual.

\smallskip

Let $\chi_{[1,|d|^{-2})}^{|d|^{-2}}$ be the $|d|^{-2}$-periodic
extension to the real line of the restriction of
$\chi_{[1,|d|^{-2})}$ to $[0,|d|^{-2})$. Then,

\centerline{$1-\sum_{m\in
  \Z}|d|^{2}\langle \chi_{[0,1)},E_{|d|^{2}m}\chi_{[0,1)}\rangle E_{|d|^{2}m}=\chi_{[1,|d|^{-2})}^{|d|^{-2}}$.}

The function
$f=\sum_{i=1}^{N-|d|^{2}N}c_{i}\chi_{[1+\frac{i-1}{|d|^{2}N},1+\frac{i}{|d|^{2}N})}$
belongs to the Wiener space, i.e., $\sum_{k\in
Z}\|f\chi_{[k,(k+1))}\|_\infty<\infty.$ As a consequence of
\cite[Proposition 8.5.2]{Christensen (2003)},
$\{E_{|d|^{2}mN}T_{n}f\}_{m,n\in \Z}$ is a Bessel sequence.

We have $h=g+f(1-\sum_{m\in\Z}|d|^{2}\langle
\chi_{[0,1)},E_{|d|^{2}m}\chi_{[0,1)}\rangle E_{|d|^{2}m})$ and, by
\cite[Proposition 9.3.8]{Christensen (2003)},
$\{E_{|d|^{2}m}T_{n}h\}_{m,n\in \Z}$ is a dual frame of
$\{E_{|d|^{2}m}T_{n}g\}_{m,n\in \Z}$.

\smallskip

The restriction to the intervals of the form $[n,n+1)$ of any
function in $W_{0}$ is $\frac{1}{|d|^{2}N}$-periodic with $|d|^{2}N$
periods. So $h \notin W_{0}$ and consequently, $V_{0}\neq
  W_{0}$. Since $V_{i}=E_{|d|^{2}i}V_{0}$ and
$W_{i}=E_{|d|^{2}i}W_{0}$, it follows that $V_{i} \neq W_{i}$ for $0
\leq i \leq N-1$.

\smallskip

The collection $\{E_{|d|^{2}m}T_{n}h\}_{m, n \in \mathbb{Z}}$ is a
Bessel sequence for $L^2(\Real)$ and $\{E_{|d|^{2}mN}T_{n}h\}_{m, n
\in \mathbb{Z}} \subseteq \{E_{|d|^{2}m}T_{n}h\}_{m, n \in
\mathbb{Z}}$, thus $\{E_{|d|^{2}mN}T_{n}h\}_{m, n \in \mathbb{Z}}$
is a Bessel sequence for $V_{0}$. Hence it has a well defined
bounded frame operator. We have

\centerline{$\sum_{m,n\in \Z}\langle
E_{|d|^{2}m'N}T_{n'}h,E_{|d|^{2}mN}T_{n}h\rangle
E_{|d|^{2}mN}T_{n}h=\|h\|^{2}E_{|d|^{2}m'N}T_{n'}h,$}

\noindent thus $\{E_{|d|^{2}mN}T_{n}h\}_{m, n \in \mathbb{Z}}$ is a
$\|h\|^{2}$-tight frame for $V_{0}$. Since $E_{|d|^{2}i}$ is a
unitary operator, $\{E_{|d|^{2}(mN+i)}T_{n}h\}_{m, n \in
\mathbb{Z}}$ is a $\|h\|^{2}$-tight frame for $V_{i}$, $0 \leq i
\leq N-1$.

\smallskip

Now we can conclude, by Theorem~\ref{T fusion frame system dual},
that $(\mathcal{V},1)$ is a dual fusion frame of $(\mathcal{W},1)$,
and it is not the canonical dual.

\end{example}

\section*{Acknowledgement}

The authors thank the valuable comments and suggestions of the
referee that improved the presentation of the paper. S. B. Heineken
thanks for the hospitality during her visit at the Departamento de
Matem\'atica of FCFMyN, UNSL. This work was supported by grants
UBACyT 2011-2014 (UBA), PICT-2007-00865 (FonCyT, ANPCyT) and
P-317902 (UNSL).

% ----------------------------------------------------------------

\end{document}